\documentclass[10pt, reqno]{amsart}
 \usepackage{hyperref, mathtools}
 \usepackage[skip=4pt]{caption}
 \usepackage{amssymb,amsmath,graphicx,amsthm}
\usepackage{color,xcolor}
\definecolor{darkred}{rgb}{1,0,0} %can change the intensity in [0,1]
\definecolor{darkgreen}{rgb}{0,0.8,0}
\definecolor{darkblue}{rgb}{0,0,1}

\hypersetup{colorlinks,
linkcolor=darkblue,
filecolor=darkgreen,
urlcolor=darkred,
citecolor=darkgreen}

 \def\bt{\begin{theorem}}
 	\def\el{\end{lemma}}
 \def\bl{\begin{lemma}}
 	\def\et{\end{theorem}}
 \def\bp{\begin{proposition}}
 	\def\ep{\end{proposition}}
 \def\bd{\begin{definition}}
 	\def\ed{\end{definition}}
 \def\br{\begin{remark}}
 	\def\er{\end{remark}}
\def\Label#1{\label{#1}}

\def\dib{\bar\partial}

 \def\R{{\mathbb R}}

 \def\C{\mathbb C}

\def\S{{\mathbb S}}
\def\S1{{\mathbb S^1}}

 \def\dib{\bar\partial}
 \def\label#1{\label{#1}{\bf(#1)}~}

 \numberwithin{equation}{section}

 \theoremstyle{plain}
 \newtheorem*{theorem*}{Theorem}

 \newtheorem{theorem}{Theorem}[section]

 \newtheorem{lemma}[theorem]{Lemma}

 \newtheorem{proposition}[theorem]{Proposition}

 \theoremstyle{definition}

 \newtheorem{definition}[theorem]{Definition}

 \theoremstyle{remark}

 \newtheorem{remark}[theorem]{Remark}

 \newcommand{\p}{\partial}

 \newcommand{\dbar}{\bar\partial}

 \newcommand{\dbarb}{\bar\partial_b}
 \theoremstyle{plain} % just in case the style had changed

 \dedicatory{To the memory of Nicholas Hanges}

\begin{document}
 	
 	\title{A note on the closed range of $\bar\partial_b$ on $q$-convex manifolds  }
\author{Luca Baracco}
\address{Dipartimento di Matematica Tullio Levi-Civita, Universit\`a di Padova, via Trieste 63, 35121 Padova, Italy}
\email{baracco@math.unipd.it}
\author{Alexander Tumanov}
\address{Department of Mathematics, University of Illinois at Urbana-Champaign, 1409 West Green
Street, Urbana, IL 61801, USA}
\email{tumanov@uiuc.edu}
 		
\thanks{Research of the second author is partially supported by Simons Foundation grant.}
 
\maketitle

 \begin{abstract} We prove that the tangential Cauchy-Riemann operator $\dbarb$ has closed range on Levi-pseudoconvex $CR$ manifolds that are embedded in a $q$-convex complex manifold $X$. Our result generalizes the known case when $X$ is a Stein manifold (in particular, when $X=\C^n$).
\end{abstract}

%\subjclass[2010]{Primary 	53D05.
%Secondary 53C15}
%\keywords{Quaternions, Hartogs separate regularity.}
  	
  %\maketitle
  %\maketitle

 %***********************************
 % Introduction
 %***********************************
% \tableofcontents 	
 	\section{Introduction}
The $\bar\partial$ operator in complex analysis is very important because it is involved in the explanation of many phenomena concerning holomorphic functions and maps. The solvability of $\bar\partial$ on forms gives a characterization of domains of holomorphy and of Stein manifolds. Moreover, the regularity of solutions of the $\bar\partial$-equation gives a different perspective on the regularity up to the boundary of biholomorphisms. Similarly, the study of $\bar\partial_b$, which is the restriction to a hypersurface of $\bar\partial$, is tightly connected with geometric aspects of $CR$ manifolds, like embeddability in $\C^n$.

The natural environment to study these differential operators is the space of forms with $L^2$ coefficients. In these spaces the operators $\bar\partial$ and $\bar\partial_b$ are naturally defined in the distribution sense. If $\Omega$ is a domain in $\C^n$, then for any $d<n$ we have a closed densily defined unbounded operator
$\bar\partial_d : L^2_{(0,d)}(\Omega)\longrightarrow L^2_{(0,d+1)}(\Omega)$.
It is easy to check that $\bar\partial_{j+1}\circ\bar\partial_{j}=0$, and thus $(L^2_{0,j}(\Omega),\bar\partial_{j})_{j=0,...,n}$ is a complex. Similarly, for $\bar\partial_b$ the same hold. Forms in the kernel of $\bar\partial$ (resp. $\dbarb$) are called $\dbar$-closed (resp. $\dbarb$-closed), and those in the image $\dbar$-exact (resp. $\dbarb$-exact). In the $\dbarb$ problem, one is given a $(0,q)$-form $f$ that is $\dbarb$-closed and wants to find a $(0,q-1)$-form $u$ such that $\dbarb u=f$ (eventually with regularity requirements on the solution). The question has been studied and solved for $\dbar$ by Kohn and H\"ormander for pseudoconvex domains in $\C^n$. For $\dbarb$, the first result of this kind was proved for strictly pseudoconvex hypersurface-type $CR$ manifolds.

The starting point in tackling these problem is proving that the range of $\dbar$ ($\dbarb$) is closed, and this is done starting from the Kohn-Morrey-H\"ormander identity. The main difference between the two situations is that in dealing with $\dbarb$ one need to control a mixed term that appears when integrating by parts. This mixed term involves the derivative of the coefficients of the form in the so called ``totally real" direction. The presence of this term is what makes the closed range for $\dbarb$ harder to check. For strictly pseudoconvex manifolds, the fact that $\dbarb$ has closed range was been proved by Kohn (indeed he proved closed range not only in $L^2$, but also in Sobolev spaces). Shaw \cite{S85bis} and Boas-Shaw \cite{BS86} proved closed range for $\dbarb$ on boundaries of pseudoconvex domains in $\C^n$. Their method does not generalize to domains in manifolds. Kohn \cite{K86}, using microlocal analysis, proved closed range of $\dbarb$ on boundaries of pseudoconvex domains in Stein manifolds.
Similar, from the complex point of view, to the boundaries of domains are the $CR$ manifolds of hypersurface type (see Definition \ref{D3}).
%A $CR$ manifold of hypersurface type in $\C^n$ is a real, odd dimensional submanifold such that its complex tangent bundle has codimension $1$ in other words these are manifolds which look like hypersurfaces.
Nicoara \cite{N06} proved closed range on pseudoconvex $CR$ manifolds of hypersurface type in $\C^n$ whose real dimension is larger than $3$. Finally, Baracco \cite{B12} proved the $3$ dimensional case using the technique of Kohn and a desingularization argument of complexification.

The pseudoconvex case in Stein manifolds is well understood.
Most of the techniques adopted rely on the existence of a plurisubharmonic weight. This is indeed a distinctive tract of Stein manifolds. Yet there are many other important manifolds that one encounter in complex analysis which do not have a plusubharmonic weight like for instance the compact manifolds. In the attempt to bridge the gap between these two extreme cases in \cite{AG62} the notion of completely $q$-convex manifolds is introduced. Roughly speaking a complex manifold $X$ of dimension $n$ is said to be completely $q$-convex  if $X$ is endowed with an exhaustion function which has a controlled number of positive Levi eigenvalues (namely greater than $q+1$ see Definitions \ref{D1} and \ref{D2} ).
Since their introduction these manifolds have been intensively studied. Most of the main tools in complex analysis can be considered in this new setting (see \cite{D,HL88,O18} and the references therein) and the main difficulty is of course the lack of convexity of the exhaustion funcion.
In this paper we want to extend further the results on the range of $\dbarb$ on pseudoconvex manifolds of hypersurface type when these are contained in a completely $q$-convex manifold. We will assume the existence of a $q+1$-convex weight defined only around $M$.  Here is the statement of our main result.
\bt \label{closedrange1} Let $X$ be a complex manifold and $M\subset\subset X$ a smooth compact, pseduconvex-oriented CR submanifold of hypersurface type of dimension $2p-1$. Assume that there exists a $(q+1)$-convex function $\phi$ defined in a neighborhood of $M$ in $X$. Then
\[\dib_b\colon  Dom(\dib_b)_{r,s-1}\longrightarrow L^2_{r,s}(M)\] has closed range if $n-q\le s \le p+q-n$ and $p>2(n-q)$. Moreover, if $X$ is completely $q$-convex, then the same conclusion holds for $p=2(n-q)$.
\et
The paper is organized as follows. In Section \ref{S2} recall some basic definitions and we show how to realize a $CR$ manifold  of hypersurface type as the boundary of a complex manifold $Y$. In Section \ref{S3} we present the proof of Theorem \ref{closedrange1}.

This paper was written for a special volume of CASJ dedicated to the memory of Nicholas Hanges.
He did pioniering work on propagation of holomorphic extendibility of CR functions, and we use related results in this paper. We will remember him as a prominent mathematician,
a wonderful person, and a great colleague.
%We start by recalling some definitions.
\section{Definitions and construction of a partial complexification}\label{S2}
  Let $X$ be a complex manifold of dimension $n$ endowed with a Hermitian product, which we denote by $(\cdot,\cdot)_p :T^{(1,0)}_pX \times T^{(1,0)}_pX\rightarrow \C$.  We first extend this scalar product  to forms in the following way. Let $L_1,...,L_n$ be a local orthonormal basis of $(1,0)$-vector fields, and let $\omega_1,...,\omega_n$ be the dual basis of $(1,0)$-forms. We define the scalar product on $(1,0)$ forms by declaring that $\omega_1,...,\omega_n$ is an orthonormal basis and we will denote this product again by $(\cdot,\cdot)_p$. Such product do not depend on the choice of the basis $L$. For forms of higher bi-degree, say $(r,s)$ we proceed in the same way by declaring that $\omega_I\wedge \overline{\omega}_J$ is an orthonormal basis where $I$ and $J$ are multi-indexes of lenght $r$ and $s$ respectively.

  The volume form of $X$ is thus given by $dV=\frac1{i^n}\omega_1\wedge\bar\omega_1\wedge\dots\wedge\omega_n\wedge \bar\omega_n$.
If $\Omega\subset X$ is a relatively compact open subset with smooth boundary we define the space $L^2(\Omega)$ of square integrable functions on $\Omega$ as the set of all complex-valued measurable functions $f$ such that
$$\int_\Omega |f|^2 dV <\infty  .$$
$L^2(\Omega)$ is a Hilbert space with scalar product given by $$\langle f,g\rangle:= \int_\Omega f\,\bar g\, dV .$$
We extend this definition to forms and define for two integers $0\le r,s\le n$, the space $ L^2_{(r,s)}(\Omega)$ which is the space of forms $f$ such that $f$ can be written locally as $\underset{{|I|=r} {|J|=s}}\sum f_{IJ}\omega_I\wedge\bar\omega_J$ where $f_{IJ}$ are measurable functions and such that $\int_{\Omega} (f,f)_p dV(p) <\infty$.
The scalar product is defined similarly by $$\langle f,g\rangle =\int_\Omega (f, g)_p  dV(p).$$
Let $\mathcal{D}(\overline{\Omega})$ be the space of functions on $\overline{\Omega}$ which are smooth up to the boundary and let  $\mathcal{D}_{r,s}(\overline{\Omega})$ be the corresponding space of $(r,s)$-forms with coefficients in $\mathcal{D}(\overline{\Omega})$. On these spaces, the operator $\dib_{(r,s)}$ is defined in the usual way as
$$ \dib_{r,s}:\mathcal{D}_{r,s}(\Omega) \longrightarrow \mathcal{D}_{r,s+1}(\Omega) $$
In local coordinates we have
$$\dib f(z)= \sum_{IJj}\frac{\p}{\p_{\bar z_j}} f_{IJ} \,d\bar{z}_j\wedge dz_I\wedge d\bar z_J .$$
If instead of a local coordinate system we use a local system of orthonormal vector fields then we have
\begin{equation}\label{f1}
 \dib f(z)= \sum_{I,J,j} \overline{L}_j f_{IJ}\,\bar{\omega}_j\wedge \omega_I\wedge \bar{\omega}_J +\dots=A(f)+\dots
\end{equation}
where dots stand for terms which do not involve derivatives of $f$ and $A$ denote the operator formed with all the terms containing the derivatives of $f$.
Since $\mathcal{D}_{r,s}(\overline{\Omega})$ is dense in $L^2_{(r,s)}(\Omega)$, we consider the maximal closed extension of $\dib$ (still called $\dib$), and its $L^2$-adjoint $\dib^*$. Particularly useful is the formal adjoint operator $\vartheta$ which on smooth compactly supported forms is characterized by the property $\langle \vartheta f,g\rangle=\langle f,\dib g\rangle$. In a local frame we have
$$
 \vartheta f = \sum_{IK} \sum_j L_j f_{I,jK} \, \omega_I \wedge\overline{\omega}_K+ \dots
$$
where dots stand for terms that do not contain derivatives of $f$.
If $\phi:X \rightarrow \R$ is a continuous functions we define the weighted Hermitian product with weight $\phi$ in the following way
$$ \langle f,g\rangle_\phi := \int_\Omega e^{-\phi}(f,g)dV .$$
The corresponding norm will be denoted by $\|\cdot\|_\phi$ and the adjoint of $\dib$ in this product will be denoted by $\dib^*_\phi$.
The formal adjoint operator $\vartheta_\phi$ in a local frame is
\begin{equation}\label{f2}
 \vartheta_\phi f:= (-1)^{r-1} \sum_{IK}\sum_j \delta^\phi_j f_{I,jK} \,\omega_{I}\wedge \bar{\omega}_K +\dots=B(f)+\dots
\end{equation}
where $\delta^{\phi}_j u=L_j(u)-L_j(\phi)u$ and dots stand for terms without derivatives of $f$ and $B$ is the operator defined by the summation in the term in the middle of \ref{f2}.
In the sequel we shall use weighted scalar product with weights of the form $t\lambda$, where $\lambda$ is a convenient function and $t$ a real parameter. When this choice is made we shall indicate the corrisponding scalar product with $\langle \cdot,\cdot\rangle_t$ and similarly for $\|\cdot\|_t,\ \dbar^*_t, \vartheta_t,\ \delta^t_j$.
Let $\phi:X\rightarrow \R$ be a smooth function. We denote by $\partial\bar\partial \phi$ the $(1,1)$-form defined at every point $z\in X$ by
$$ \partial\bar\partial \phi(z)=\sum_{i,j=1}^n\partial_{z_i}\p_{\bar{z}_j}\phi(z)\,dz_i\wedge d\bar z_j, $$
where $z_1,\dots,z_n$ are local coordinates for $X$ at $z$.  If an orthonormal basis $\omega_i$ of $(1,0)$-forms has been chosen, then we shall also write
$$ \p\dib \phi=\sum_{ij}\phi_{ij}\,\omega_i\wedge \bar\omega_j .$$

The form $\p\dib \phi$ defines a Hermitian form, called the {\em Levi form}, on the holomorphic tangent bundle $T^{1,0}X$ of $X$. The Levi form will also be denoted by $\partial\bar\partial\phi$, and is defined in the following way. For $z\in X$ and vectors $X=\sum a_i\p_{z_i}$ and $Y=\sum_i b_i \p_{z_j}$ in $T_z^{1,0}X$, we have
$$ \p\bar\p \phi(z)(X,\overline{Y})=\sum_{i,j=1}^n a_i\bar{b}_j\,\partial_{z_i}\p_{\bar{z}_j}\phi(z) .$$
\bd \label{D1}We say that $\phi$ is {\em $q$-convex} if the Levi form $\partial\bar\partial \phi$ %$(\phi_{ij})(z):=(\partial^2_{z_i\bar z_j}\phi (z))$
has at least $q$ positive eigenvalues.
\ed
\bd \label{D2}We say that $X$ is {\em  completely $q$-convex} if there exists a smooth exhaustion function $\phi:X\rightarrow \R$ which is $(q+1)$-convex.
\ed
Let $J:TX\rightarrow TX$ be the standard complex structure induced by the multiplication by $i$ and let $M$ be a real submanifold of $X$.
\bd\label{D3} The complex tangent space to $M$ at a point $z\in M$ is the subspace
$$T^\C_z M :=T_z M\cap JT_z M .$$
We will say that $M$ is a CR manifold if $T^\C_zM$ has constant dimension. The bundle so formed is called the complex tangent bundle of $M$ and is denoted by $T^\C M$. We say that $M$ is of hypersurface type if $\frac{TM}{T^\C M}$ has rank $1$.
\ed
Let $M$ be a smooth, compact CR submanifold of $X$ equipped with the induced CR structure $T^{1,0}M=\C TM\cap T^{1,0}X$. The De Rham exterior derivative induces a complex on skew-symmetric antiholomorphic forms on $M$. We denote such complex by $\dib_b$. Assume that $M$ is of hypersurface type. Hence the complexified tangent bundle $\C TM$ is spanned by $T^{1,0}M$, its conjugate $T^{0,1}M$ and a single additional vector field $T$. We can assume $T$ to be purely imaginary, that is, satisfying $\overline{T}=-T$.
%Our main result is the following:
%\bt \label{closedrange1} Let $X$ be a complex manifold and $M\subset\subset X$ a compact, pseduconvex-oriented CR submanifold of hypersurface type of dimension $2p-1$. Assume that there exists a $(q+1)$-convex function $\phi$ defined in a neighborhood of $M$ in $X$. Then
%\[\dib_b\colon  Dom(\dib_b)_{r,s-1}\longrightarrow L^2_{r,s}(M)\] has closed range if $n-q\le s \le p+q-n$ and $p>2(n-q)$. Moreover, if $X$ is completely $q$-convex, then the same conclusion holds for $p=2(n-q)$.
%\et

 Let $\eta$ be a purely imaginary $1$-form which annihilates $T^{1,0}M\oplus T^{0,1}M$ and normalized so that $\langle \eta,T\rangle=-1$. The manifold $M$ is {\em orientable} if there exists a global 1-form section $\eta$ (or vector field $T$) and is {\em pseudoconvex} if the hermitian form defined on $T^{1,0}M$ by $d\eta (X,Y)=\langle d\eta, X\wedge \overline{Y}\rangle$ is positive semidefinite. We say that $M$ is {\em pseudoconvex-oriented} if both properties are satisfied at the same time.

A CR curve $\gamma$ on $M$ is a real curve such that $T\gamma\subset T^\C M$. A CR orbit is the union of all piecewise smooth CR curves issued from a point of $M$. We denote by  $\mathcal{O}(z)$ the $CR$ orbit of a point $z\in M$, and we say that a set $S$ is $CR$ {\em invariant} if $\mathcal{O}(z)\subset S$ for all $z\in S$. By Sussmann's Theorem \cite{MP06}, the orbit $\mathcal{O}(z)$ has the structure of an immersed variety of $X$.
Following  \cite{B12} and \cite{J95}, we prove that the manifold M in question consists of a single orbit. The difference with
\cite{B12,J95} is that instead of holomorphic coordinate functions that are not available here, we use the given (q+1)-convex function.
\bp
\Label{p2.1}
Let $X$ be a complex manifold and $M$ a smooth, compact, connected CR submanifold of hypersurface type. Let $\phi$ be a $(q+1)$-convex function defined on a neighborhood of $M$ in $X$. Assume that the dimension of $M$ is $2p-1$, with $p>n-q$. Then $M$ consists of a single CR orbit.
\ep
\begin{proof} Let $S\subset M$ be a closed, non empty $CR$ {\em invariant} subset of $M$. Since $M$ is compact we can assume that $S$ is the smallest of such sets i.e. that it doesn't contain any smaller closed non-empty $CR$ invariant subset. We will now prove that $S=M$. Assume by contradiction that $S\neq M$. For a point $x\in S$, we have only two possibilities: either $x$ is minimal in the sense of Tumanov (that is, the local $CR$ orbit of $x$ contains a neighborhood of $x$ in $M$) or there exists a complex manifold of dimension $p-1$ contained in $M$ that passes through $x$. No point $x\in S$ can be minimal in the sense of Tumanov, otherwise the set $S\setminus \mathcal{O}(x)$ would be proper, closed, $CR$ invariant, and strictly smaller than $S$. Hence $S$ is foliated by complex manifolds of dimension $p-1$. Since $S$ is compact, there exists a point $\bar x$ of $S$ where $\phi$ achieves its maximum value. In particular, $\phi$ has a maximum on the complex leaf passing through $\bar x$. This is impossible, because the Levi form of $\phi$ has at least one positive direction in the complex tangent space of $M$. We have reached a contradiction, thus proving that $S=M$. Let now $x\in M$ be the point where $\phi$ reaches its maximum. By the same reasoning as above, we can rule out the case in which there exists a complex manifold of dimension $p-1$ contained in $M$ that passes through $x$. Hence $x$ must be a minimal point in the sense of Tumanov. It follows that $\mathcal{O}(x)$ is open in $M$. Since $M$ is the smallest closed CR invariant subset element, we conclude that $\mathcal{O}(x)=M$.
\end{proof}
\begin{proposition}
\Label{t2.1}
Let $M\subset\subset X$ be a smooth, compact, connected, pseudoconvex-oriented CR manifold of hypersurface type of dimension $2p-1$. Let $\phi$ be a $(q+1)$-convex function defined on a neighborhood of $M$ in $X$, with $q>n-p$. Then $M$ is endowed with a partial one-sided complexification in $X$. That is, there exists a complex manifold $Y\subset\subset X$ which has $M$ as the smooth connected component of its boundary on the pseudoconvex side.
\end{proposition}
\begin{proof}
The proof follows closely \cite{B12}, since some geometric details will be needed in the next section we repeat the proof here for the reader's convenience.
The set of points of $M$ for which there exists a neighborhood where $M$ has a one-sided, positive, partial complexification is obviously open. We show that this set is also non-empty and closed. Let $z_0\in M$ be a point where there is no $(p-1) $ dimensional complex submanifold $S\subset M$ (see the proof of Proposition \ref{p2.1}). Consider a local coordinate patch $U\subset\C^n$ of $X$ at $z_0$ in which the projection $\pi_{z_0}:\C^n\to T_{z_0}M+iT_{z_0}M\cong \C^p$ induces a diffeomorphism between $M$ and $\pi_{z_0}(M)$. Since $\pi_{z_0}(M)$ is part of a mininimal  and pseudoconvex hypersurface, then $(\pi_{z_0}|_{M})^{-1}$ extends holomorphically to the pseudoconvex side $\pi_{z_0}(M)^+$ by \cite{Tr86} and \cite{T88}. Moreover, the map $(\pi_{z_0}|_{M})^{-1}$ parametrizes a one-sided complex manifold which has a neighborhood of $z_0$ in $M$ as its boundary.
By global pseudoconvexity and by uniqueness of holomorphic functions having the same trace on a real hypersurface, one-sided complex  neighborhoods glue together into a complex neighborhood of a maximal open subset $M_1\subset M$. This is indeed also closed. In fact, let $z_1\in \overline{M}_1$. Since $M$ consists of a single CR orbit by Proposition~\ref{p2.1}, then $z_1$ is connected to any other point of $M_1$ by a piecewise smooth CR curve $\gamma$. The statement now follows from the lemma below whose proof can be found in \cite{B12} and \cite{T94}.
\end{proof}
\bl
\Label{l2.1}
Let $M\subset\subset X$ be a smooth, pseudoconvex-oriented CR manifold of hypersurface type. Let $\gamma$ be a piecewise smooth CR curve connecting two points $z_0$ and $z_1$ of $M$. If $M$ has complex extension in direction $+JiT(z_0)$ at $z_0$, then $M$ also has complex extension in direction $+JiT(z_1)$ at $z_1$.
\el
%\begin{proof}
%Let $\xi$ be the end-point of the curve $\gamma$ for complex extension, and let $\pi_\xi:\C^n\to T_\xi M+iT_\xi M$. Then $\pi_\xi(M)$ is part of a complex hypersurface and $\pi_\xi(\gamma)$ is a CR curve. Now, if there is no germ of a complex hypersurface $S$ with $\xi\in S\subset M$, then $(\pi_\xi|_M)^{-1}$ extends holomorphically from $\pi_\xi(M)$ in direction $+\pi'_\xi(JiT)$. On the other hand, assume that such $S$ exists. Since $\pi_\xi(\gamma)$ is a CR curve, it must seat inside $S$ in a neighborhood of $\pi_\xi(\xi)$.
%By the Hanges-Tr\`eves Theorem \cite{HT83}, the extension of $(\pi_\xi|_M)^{-1}$ to $+\pi'_\xi(JiT)$ propagates along $S$ beyond $\pi_\xi(\xi)$. Hence $\xi=z_1$.
%\end{proof}
\begin{remark} In the proof of Proposition \ref{p2.1} we have used \cite{Tr86} and \cite{T88} to build a one sided complexification of $M$ near minimal points. The results in \cite{Tr86} and \cite{T88}, however, do not specify on which side of $M$ this complexification lies. Our hypotheses on the pseudoconvexity of $M$ assures that the side of the complexification at minimal points is the pseudoconvex side of $M$, namely the side pointed by $JiT$.
\end{remark}
%\begin{remark} We remark here that for the existence of $Y$ the only necessary hypotheses is the existence of a $q+1$ convex function in a neighborhood of $M$.
%\end{remark}

%\bt \label{closedrange1} Let $X$ be a complex manifold and $M\subset\subset X$ a compact, pseduconvex-oriented CR submanifold of hypersurface type of dimension $2p-1$. Assume that there exists a $(q+1)$-convex function $\phi$ defined in a neighborhood of $M$ in $X$. Then
%\[\dib_b\colon  Dom(\dib_b)_{r,s-1}\rightarrow L^2_{r,s}(M)\] has closed range if $n-q\le s \le p+q-n$.
%\et
\section{Proof of the main result}\label{S3}
We follow the same proof as in \cite{B12}, and first prove a closed range theorem for $\dib$ on an annulus-like domain:

\begin{proposition} \label{closedrange} Let $M$ and $\phi$ be as in Proposition \ref{t2.1} and assume further that $p>2(n-q)$. Then there exists a complex sub-manifold $Y$ of $X$ of dimension $p$ with smooth boundary $\p Y$ such that $\p Y= M\cup M_2$, where $M_2$ is $CR$ of hypersurface type. Moreover, if $\dib$ is the Cauchy Riemann operator on $Y$, for a suitable weight function $\lambda$ we have,  for all $f\in Dom(\dib^*_t)_{r,s}\cap C^{\infty}_{(r,s)}(\overline{Y})$, that
\begin{equation}\label{anulusest}
t \| f\|^2_t \le C(\| \dib f\|_t^2 +\| \dib^*_t f\|_t^2) +C_t\|f\|^2_{-1} \quad \forall s\ge n-q, \  s\le p+q-n-1,\ \forall t>0 .
\end{equation}
\end{proposition}
\begin{proof} First we equip the manifold $X$ with an Hermitian product in such a way that if $\phi_1(z)\le \dots \le\phi_n(z)$ are the ordered eigenvalues of $\p\dib\phi(z)$, then
$$\phi_1(z)+\dots +\phi_{n-q}(z)>c>0$$ for some constant $c$ and for every $z$ in a neighborhood of $M$. This is possible because $\phi_{n-q}>0$ by the $q+1$-convexity of $\phi$ (see \cite[Lemma IX.3.1]{D}). Let $Y$ be the complex manifold constructed in Proposition \ref{t2.1}.
Note that by the construcion made there we have that for any point $z_0 \in M$ there exists a local coordinate patch $U\subset\C^n$ of $X$ at $z_0$ in which the projection $\pi_{z_0}:\C^n\to T_{z_0}M+iT_{z_0}M\cong \C^p$ induces a diffeomorphism between $M\cap U$ and $\pi_{z_0}(M)$ and is a biholomorphism between $Y\cap U$ and $\pi_{z_0}(M)^+ \cap V$ where $V$ ia a convenient neighborhood of $\pi(z_0)$ in $\C^p$. We shall use $\pi_{z_0}$ as a local coordinate chart of $Y$.
Let $\rho: Y\rightarrow \R$ be a smooth, non-negative function such that $\rho =0$ on $M$, and $d \rho \neq 0$ on $M$. Since $M$ is pseudoconvex, then the negative eigenvalues of the Levi form of $-\log(\rho)$ are bounded from below. In fact, in the local chart as above around $z_0\in M$ we have that $\rho =hd_M$ where $h$ is a positive and non vanishing smooth function and $d_M$ is the Euclidean distance in $\C^p$ of $\pi_{z_0}(z)$ from $\pi_{z_0}(M)$. Then by Oka's lemma we have that $-\log(d_M)$ is plurisubharmonic then $-\log(\rho)=-\log(h)-\log(d_M)$.
The Levi form $\p\dib\phi$ restricted to $T^{1,0}Y$ has at least $p+q+1-n$ positive eigenvalues. Let $\varphi:= -\log(\rho) +C\phi$. If the constant $C$ is large enough, then $\partial\bar\partial\varphi$ restricted to $T^{1,0}Y$ has at least the same number of positive eigenvalues as $\p\dib\phi$. Since $d\rho \neq 0$ , then $d\varphi \neq 0$ in a neighborhood of $M$ in $Y$. In particular, the subset of $Y$ defined by the equation $\rho e^{-C\phi} =e^{-K}$ is a regular hypersurface if $K$ is large enough. We call it $M_2$. The Levi form of $M_2$ has at least $p+q-n$ positive eigenvalues. We consider now the annulus-like domain, which we call again $Y$, defined by $\varphi >K$. The boundary of $Y$ consists of the two connected components $M$ and $M_2$. We exploit \cite[Lemma IX.3.1]{D} once again to choose the Hermitian metric on $X$ in a neighborhood of $M_2$ so that the following is true: if $\varphi_1(z)\le\dots \le\varphi_{p-1}(z)$ are the ordered eigenvalues of the tangential Levi form $\p\dib\varphi$ (i.e. restricted to $T^{(1,0)}M_2$) at a point $z\in M_2$, then the sum of any $n-q$ of such eigenvalues is strictly positive.
We now follow \cite[page 260]{S85}. First we choose the weight function.
Let $\lambda\in C^2(\overline{Y})$ be a function such that $\lambda=\phi$ in a neighborhood of $M$ and $\lambda =-\varphi$ in a neighborhood of $M_2$.
It is enough to prove \refeq{anulusest} locally in a neighborhood of the boundary $\p Y$.
Let $z\in M$ and $U_z$ a small neighborhood of $z$ in $X$. Choose a local system of orthonormal holomorphic vector fields $L_1,\dots,L_p$ tangent to $Y$ such that $L_j(\rho)=-\delta_{jp}$, and let $\omega_1,...,\omega_p$ be the corresponding dual frame.  Following \cite{S85} and \cite{H} we have the following Kohn-H\"ormander-Morrey type formula:
\begin{equation}\label{seconda}
\begin{split}
\|\bar\p f\|_t^2+\|\bar\p^*_t f\|_t^2 &=\sum _{I,J} \sum_j  \| \bar L_jf_{IJ}\|_t^2 +t\,\underset{I,K}{\sum}^{\prime}\sum_{j,k} (\lambda_{jk} f_{I,jK},f_{I,kK})_t \\
&-\underset{IK}{\sum}^{\prime}\sum_{j,k<p} \int_{M\cap U_z} \langle \rho_{jk} f_{I,jK},f_{I,kK}\rangle_t dS +R(f)+E(f) .
\end{split}
\end{equation}
where $R(f)+E(f)$ are terms as in \cite[page 263]{S85} that arise when manipulating $\|\bar\p f\|_t^2+\|\bar\p^*_t f\|_t^2$ .
In fact when we replace the terms inside the norms with the terms defined in equations \ref{f1} and \ref{f2} we consider first the terms that contain squares of derivatives of $f$ and we group all the other terms in the term indicated by $R(f)$. So $R(f)$ contains only terms that can be estimated, uniformly in $t$, by $(\|A(f)\|_t+\|B(f)\|_t+\|f\|_t)\|f\|_t$.
The next step in proving \ref{seconda}  is to turn the $L$ derivatives of $f$ in $\|B(f)\|^2_t$ into $\overline{L}$ derivatives by integration by parts. In doing so some new terms arise. These terms that can be estimated uniformly in $t$ by $(\|\overline{L}(f)\|_t \|f\|_t$ where $\|\overline{L}(u)\|_t^2=\sum_j \|\overline{L}_j(u)\|_t^2 +\|u\|^2_t$ and we indicate these terms with $E(f)$.
By the pseudoconvexity of $M$, the last boundary integral in \eqref{seconda} is positive, and can therefore be dropped. Near $M$ we have that $\lambda=\phi$, and thus $\lambda_{jk}=\phi_{jk}$. Moreover, by the choice of the Hermitian product, the sum of any $n-q$ eigenvalues of the Levi form of $\phi$ is strictly positive. Hence the second term on the left side of \refeq{seconda} is greater than
$t c \|f\|^2_t$ if the antiholomorphic degree $s$ of $f$ is $s\ge n-q$. The terms $R(f)$ and $E(f)$ can be estimated using the first two term on the right hand side of \ref{seconda}.
Let now $z$ be a point of $M_2$. We start with the same formula \refeq{seconda} where $M$ is replaced by $M_2$ and $\rho$ is replaced by a defining equation of $M_2$ which is of the form $\rho_2=(\lambda +K)h$ where $h$ is a positive function such that $|d\rho_2|=1$ on $M_2$. After integrating by parts the terms of type $\| \bar L_j f_{IJ}\|^2_t$ for $j<p$ we obtain
\begin{equation}\label{terza}
\begin{split}
\|\dib f\|^2_t +\|\dib^*_t f \|^2_t&=\underset{I,J}{\sum}^\prime \| \bar L_p (f_{IJ})\|^2_t +\underset{I,J}{\sum}^\prime \sum_{j<p} \| \delta^t_j f_{I,J}\|^2_t  \\
&+t\, \underset{I,K}{\sum}^\prime \sum_{j,k} (\lambda_{jk} f_{I,jK},f_{I,kK})_t -t\,\underset{I,J}{\sum}^\prime \sum_{j<p} (\lambda_{jj}f_{I,J},f_{I,J})_t \\
&+\underset{I,K}{\sum}^\prime \sum_{j,k<p} \int_{M_2\cap U}h\lambda_{jk}f_{I,jK}\bar f_{I,kK} e^{-t\lambda}dS \\&-\underset{I,J}{\sum}^\prime \sum_{j<p} \int_{M_2\cap U}h\lambda_{jj}|f_{I,J}|^2 e^{-t\lambda} dS +E(f)+R(f).
\end{split}
\end{equation}
Here we have used the fact that the defining function of $M_2$ has essentially the same Levi form as the weight function $\lambda$.
We first examine the terms of  \refeq{terza} consisting of boundary integrals. It is not restrictive for our purposes to assume that pointwise the tangential Levi form is diagonal and let $l^\lambda_1,..., l^\lambda_{p-1}$ be its eigenvalues. Then
\begin{equation}\label{a} \underset{I,K}{\sum}^\prime \sum_{j,k<p} \lambda_{jk}f_{I,jK}\bar f_{I,kK}=\underset{I,J}{\sum}^\prime \sum_{j\in J}l^\lambda_j |f_{I,J}|^2.
\end{equation}
After subtracting from \eqref{a} the term coming from the second boundary integral, which is just the trace of the tangential Levi form, we obtain
\begin{equation}\label{Sub}
 -\underset{J}{\sum}^\prime \sum_{j\notin J, }l^\lambda_j |f_{I,J}|^2>\underset{J}{\sum}^\prime c |f_{I,J}|^2,
\end{equation}
where $c$ is strictly positive as soon as $|J|<p+q-n$. In a similar way we can handle the tangential terms (i.e. those for $j,k <p$ and $p\notin K,J$) in the second line of \eqref{terza}. The terms with either $p\in J,K$ or $j=p$ or $k=p$ can be handled as in \cite{S85}.
\end{proof}
Choosing $t$ large enough, we can pass from \eqref{terza} to a priori estimates of higher Sobolev order.
As a consequence using the elliptic regularization technique as was done in \cite{KN65} and \cite{K73}  we obtain that the space of harmonic forms on $Y$ which is the space $H^{r,s}_t(Y):=\ker(\bar\partial)\cap \ker(\bar\partial^*_t)$, is finite dimensional and  moreover we have a Hodge decomposition on the space of forms orthogonal to $H^{r,s}_t(Y)$,  existence and global Sobolev regularity of the $\dib$-Neumann operator $N_t$.

 We are now in position to prove Theorem \ref{closedrange1}.
\begin{proof}[Proof of Theorem \ref{closedrange1}]
By Proposition \ref{closedrange} there exists a complex submanifold $Y$ whose boundary contains $M$. We  follow \cite{K86} Paragraph 5  and we have that $\dib_b$ has closed range in degree $s$ if \ref{anulusest} holds in degree $s$ and  $p-s-1$. Therefore by Proposition \ref{closedrange} we have closed range of $\dib_b$ for $ n-q\le s\le p+q-n$ provided $p+q-n-1\ge n-q$, that is, $p> 2(n-q)$.

 If $X$ is $q$-complete, then by \cite{KS01} it is possible to extend $M$ to an analytic set $E$ whose boundary, in the sense of currents, is $M$. By the Hironaka desingularization theorem \cite{J08} there exists a manifold $\tilde{E}$ and a proper bimeromorphic map $\pi \colon \tilde{E}\rightarrow E$ such that $\pi$ is an isomorphism over the non singular part of $E$. Near $M$ we have that $E$ coincides with $Y$ at the regular points of $E$. Since $\pi$ is onto and since the regular part of $E$ is dense and connected, it follows that $\tilde{E}$ contains an isomorphic copy of $Y$ and $\pi$ is an actual diffeomorphism near the boundary.
Pulling back $\phi$ on $\tilde{E}$, we can repeat the proof of Theorem \ref{closedrange}, where the boundary of the manifold $Y$ is now just $M$. Following the same proof we conclude that the range of $\dib_b$ is closed for $p=2(n-q)$.
\end{proof}
\section{Acknowledgements}
The authors would like to thank Martino Fassina for useful comments on the manuscript and the anonymous referee whose advice improved greatly the expository quality of the paper.

 								\bibliographystyle{alpha}
 							%	\bibliography{mybibTVK}

 							\end{document}